\long\def\symbolfootnote[#1]#2{\begingroup\def\thefootnote{\fnsymbol{footnote}}
\footnote[#1]{#2}\endgroup}

\documentclass[12pt]{amsart}
\usepackage{amsfonts}
\usepackage{amsmath}

\newtheorem{theorem}{Theorem}

\newtheorem{corollary}[theorem]{Corollary}

\newtheorem{lemma}[theorem]{Lemma}

\numberwithin{equation}{section}

\begin{document}
\title{Holomorphic functions on certain  K\"{a}hler manifolds}

\author{Ovidiu Munteanu and  Jiaping Wang}
\maketitle
\begin{abstract}
We investigate Liouville theorems and dimension estimates for the space of exponentially growing
holomorphic functions on complete K\"{a}hler manifolds. While our work is  motivated by the 
study of gradient Ricci solitons in the theory of Ricci flow, the most general results we prove here
do not require any knowledge of curvature. 
\end{abstract}

\section{Introduction}

\symbolfootnote[0]{The first author partially supported by NSF grant No.
DMS-1262140 and the second author by NSF grant No. DMS-1105799}

On the complex Euclidean space $\mathbb{C}^n,$ the classical Liouville
theorem says that any bounded holomorphic function must be a constant. More
generally, any holomorphic function of polynomial growth is necessarily a
polynomial. In particular, this implies that the space of holomorphic
functions with any fixed polynomial growth order is of finite dimension. It
also shows that the ring consisting of all polynomial growth holomorphic
functions is finitely generated by the coordinate functions.

On general complete K\"ahler manifolds, it is obviously of great interest to
address these issues. Under suitable curvature assumptions, there are
satisfactory results concerning finite dimensionality of the space of
polynomial growth holomorphic functions. Indeed, on a complete K\"ahler
manifold with nonnegative Ricci curvature, by the well known result of Yau
\cite{Y1} and the fact that holomorphic functions are harmonic, any bounded
holomorphic function must be a constant. More generally, as a consequence of
the results of Colding and Minicozzi \cite{CM} and P. Li \cite {L1} on
polynomial growth harmonic functions, the space of polynomial growth
holomorphic functions of any fixed order is necessarily of finite
dimension. In fact, in a more recent work of L. Ni \cite{N}, a sharp
dimension upper bound together with a rigidity result was established for
such spaces by assuming instead that the bisectional curvature is
nonnegative. However, the more significant question of whether the ring of
polynomial growth holomorphic functions is finitely generated on a complete
K\"ahler manifold with nonnegative bisectional curvature still remains
unresolved. This question has been raised by Yau \cite{Y2} and seems to be
motivated by the uniformization conjecture \cite{Y2} that such K\"ahler
manifold is biholomorphic to $\mathbb{C}^n.$ We refer to the interesting
work of Mok \cite{M} for progress and further information.

Our purpose here is to establish some Liouville type results on K\"ahler
manifolds without involving curvature conditions. This is largely motivated
by the consideration of the so-called gradient K\"ahler Ricci solitons.
Recall that Riemannian manifold $(M, g)$ is a gradient Ricci soliton if
there exists a smooth function $f\in C^{\infty }\left( M\right),$ called the
potential function of the soliton, such that the following equation holds
true for some constant $\lambda.$
\begin{equation*}
\mathrm{Ric}+\mathrm{Hess}\left( f\right)=\lambda\,g.
\end{equation*}
Here, Ric denotes the Ricci curvature of $M$ and Hess(f) the hessian of
function $f.$ The soliton is called shrinking, steady or expanding,
respectively, if $\lambda>0,$ $\lambda=0$ or $\lambda<0.$ Ricci solitons are
simply the self-similar solutions to the Ricci flows (see \cite{CLN}). They
may also be viewed as natural generalization of Einstein manifolds. In the
case $M$ is a K\"ahler manifold, Ricci solitons are called K\"ahler Ricci
solitons. With respect to unitary frames, the defining equation for gradient
K\"ahler Ricci soliton can then be written into
\begin{eqnarray*}
R_{\alpha \overline{\beta }}+f_{\alpha \overline{\beta }}&=&\lambda\,
\delta_{\alpha \overline{\beta}} \\
f_{\alpha \beta } &=&0.
\end{eqnarray*}

The important fact to our consideration here is that the potential function $%
f$ satisfies $f_{\alpha \beta }=0$ on a gradient K\"{a}hler Ricci soliton.
This condition may be rephrased (see \cite{Ca}) as $\nabla f$ being the real part of a holomorphic vector field, or as
$J(\nabla f)$ being a Killing vector field on $M$.
As we shall see later, the existence of such a function $f$ leads to
Liouville type results without involving any curvature conditions.

In the following and throughout the paper, we denote by $r(x)$ the distance
function from $x$ to a fixed point $p$ on manifold $M.$ If the volume $%
V_p(R) $ of the geodesic ball $B_p(R)$ in $M$ satisfies

\begin{equation*}
V_p(R)\le C\, e^{a\,R}
\end{equation*}
for all $R>0,$ where $C$ and $a$ are constants, then $M$ is said to have
exponential volume growth with rate $a.$ If instead one has

\begin{equation*}
V_p(R)\le C\, (R+1)^m
\end{equation*}
for all $R>0,$ where $C$ and $m$ are constants, then $M$ is said to have
polynomial volume growth of order $m.$

Also, we denote by $E(d)$ the space of holomorphic functions $u$ on $M$ of
exponential growth rate at most $d,$ that is,
\begin{equation*}
|u|(x)\leq c\, e^{d\,r(x)}
\end{equation*}
for $x\in M.$

Similarly, $P(d)$ is the space of holomorphic functions $u$ on $M$ of
polynomial growth order at most $d,$ i.e., for some constant $c,$
\begin{equation*}
|u|(x)\leq c\, \left(r(x)+1\right)^d
\end{equation*}
for $x\in M.$

\begin{theorem}
\label{holo}Let $M$ be a complete K\"{a}hler manifold. Assume that there
exists a proper function $f$ on $M$ such that $f_{\alpha \beta }=0$ with
respect to unitary frames. Then for all $d>0,$

(a) $\dim E(d)<\infty$ if $|\nabla f|$ is bounded and $M$ has exponential
volume growth.

(b) $\dim P(d)<\infty$ if $|\nabla f|$ grows at most linearly and $M$ has
polynomial volume growth.
\end{theorem}

Here, $|\nabla f|$ is said to grow at most linearly if $|\nabla f|(x)\le
b\,(r(x)+1)$ on $M$ for some constant $b.$ Note that part (a) of the theorem
immediately implies that the space $E(d_0)$ only consists of the constant
functions for some $d_0>0.$ Indeed, one may take $d_0=1/k,$ where $k=\dim
E(1).$ To see that $\dim E(d_0)=1,$ note that $1, u, u^2,\cdots,u^k$ are
linearly independent and belong to $E(1)$ for a nonconstant function $u\in
E(d_0).$ This shows that $\dim E(1)\geq k+1,$ an obvious contradiction.
Similarly, part (b) of the theorem implies that any bounded holomorphic
function must be constant.

Let us also remark that it is possible to make the dimension estimates
explicit in Theorem \ref{holo}, as to establish these results we have used a direct argument.
 However, in both cases, the estimates depend
on the local geometry of $M$ and behavior of $f.$

As mentioned above, an important class of examples to which the theorem
applies is the gradient K\"{a}hler Ricci solitons. There, the potential
function $f$ automatically satisfies $f_{\alpha \beta}=0$ with respect to
unitary frames. 

In the case of the steady solitons, according to \cite{MS},
the volume growth is subexponential, that is, it is of exponential growth
with arbitrarily small rate. It is also well known that $\left\vert \nabla
f\right\vert$ is bounded as $S+\left\vert \nabla f\right\vert ^{2}$ is a
constant by a result of Hamilton and the scalar curvature $S$ is nonnegative by Chen \cite{C}.

Therefore, the following corollary follows from part (a) of Theorem \ref{holo}.

\begin{corollary}
Let $\left( M,g,f\right) $ be a gradient K\"{a}hler Ricci steady soliton
with proper potential function $f$. Then the space of exponential growth
holomorphic functions satisfies $\dim E(d)<\infty $ for all $d>0.$ In
particular, any subexponentially growing holomorphic function on $M$ must be a
constant.
\end{corollary}

The assumption that $f$ is proper is indeed necessary in the corollary. This
is because $\mathbb{C}^{n}$ with $f\left( x\right) =\left\langle
a,x\right\rangle +b$ is a gradient K\"{a}hler Ricci steady soliton.
Obviously, there are nontrivial polynomial growth holomorphic functions. On
the other hand, the well known example of the K\"{a}hler Ricci steady
soliton, the cigar soliton, given by
\begin{equation*}
\left( \mathbb{R}^{2},\frac{dzd\overline{z}}{1+\left\vert z\right\vert ^{2}}%
,-\ln \left( 1+\left\vert z\right\vert ^{2}\right) \right) ,
\end{equation*}%
admits nonconstant holomorphic functions $u\left( z\right) =cz.$ It can be
easily checked that $u$ grows exponentially. In this case, the potential
function $f$ is proper. The two examples show the sharpness of the corollary.
In passing, we note that by \cite {MS}, there is no nonconstant holomorphic function
with finite Dirichlet energy on any gradient  steady Ricci soliton. 

We would like to mention that under more stringent assumptions that the Ricci curvature 
is positive and the scalar curvature achieves its maximum, a gradient K\"{a}hler 
steady Ricci soliton is biholomorphic to $\mathbb{C}^n$. This result was proved by 
Chau and Tam \cite {CT} and Bryant \cite {B}. The proof in \cite {B} is by constructing
a global holomorphic coordinate system ${z_1,...,z_n}$ on $M$ directly. These coordinate 
functions are shown to be of exponential growth on $M$. It is not difficult to see from there that any holomorphic function $u$ 
of exponential growth on $M$ is necessarily a polynomial of $z_1,...,z_n$. In particular,
this implies the corollary and also that the ring of exponential growth holomorphic functions on $M$ is finitely generated.

For both shrinking and expanding gradient Ricci solitons, the gradient of
the potential function $f$ grows at most linearly. So part (b) of Theorem %
\ref{holo} leads directly to the following conclusion for gradient K\"ahler
Ricci expanding solitons.

\begin{corollary}
Let $\left( M,g,f\right) $ be a gradient K\"{a}hler Ricci expanding soliton
with proper potential function $f$ and polynomial volume growth. Then the
space of polynomial growth holomorphic functions satisfies $\dim P(d)<\infty$
for all $d>0.$
\end{corollary}

In the case of shrinking gradient Ricci solitons, stronger result is
available. Note that by scaling the metric, we may assume $\lambda=\frac{1}{2%
}.$ Now the potential function $f,$ after adding a suitable constant,
satisfies

\begin{equation*}
S+\left\vert \nabla f\right\vert ^{2}=f \text{ and \ \ }S+\Delta f=\frac{n}{2%
}.
\end{equation*}%
In addition, by Chen \cite{C},
\begin{equation*}
S\geq 0.
\end{equation*}
Furthermore, Cao and Zhou \cite{CZ} have shown that
\begin{equation*}
\frac{1}{4}\left( d\left( p,x\right) -c\left( n\right) \right) ^{2}\leq
f\left( x\right) \leq \frac{1}{4}\left( d\left( p,x\right) +c\left( n\right)
\right) ^{2}
\end{equation*}%
for any $x\in M$ and

\begin{equation*}
V_p(R)\le c(n)\,R^n
\end{equation*}
for $R\geq 1.$ Here, $p\in M$ is a minimum point for $f$ and the constant $%
c(n)$ depends only on the dimension $n$ of $M.$

In particular, one sees that $f$ is proper, $|\nabla f|$ grows at most
linearly and $M$ has polynomial volume growth of order $n.$ Applying part
(b) of Theorem \ref{holo}, one concludes that $P(d)$ is finite dimensional
for all $d>0.$ It turns out the dimension of the spaces $P(d)$ can be
estimated by a universal constant only depending on the growth order $d$ and
the dimension $n$ of the underlying manifold.

\begin{theorem}
\label{Shrinking}Let $\left( M^n,g,f\right) $ be a gradient K\"{a}hler Ricci
shrinking soliton of complex dimension $n$. Then $\dim P(d)\le C(n,d),$ a
constant depending only on $n$ and $d,$ for all $d>0.$
\end{theorem}

It is possible to obtain the constant $C(n,d)$ explicitly, as a polynomial of $d$. 
Again, it has been shown in \cite {MS} that a holomorphic function 
with finite Dirichlet energy on a gradient K\"{a}hler shrinking Ricci soliton is
necessarily a constant. 
For Theorem \ref{Shrinking} , an important example to keep in mind is the Gaussian
shrinking soliton given by $M=\mathbb{C}^{n}$ endowed with the Euclidean
metric and $f\left( z\right) =\frac{1}{4}\left\vert z\right\vert ^{2}.$
Clearly, there exist holomorphic functions of polynomial growth. It will be
interesting to see if the dimension of the space $P(d)$ is actually
maximized over the Gaussian shrinking soliton among all the gradient K\"{a}%
hler Ricci shrinking solitons. This will be an analogue to the
aforementioned result of Ni \cite{N} concerning K\"{a}hler manifolds with
nonnegative bisectional curvature.

Our technique here does not seem to allow us to address the issue of whether
the ring of all polynomial (or exponential) growth holomorphic functions is
finitely generated. In view of Theorem \ref{holo}, one can speculate that
this is the case for the ring of all exponential growth holomorphic
functions and for the ring of all polynomial growth holomorphic functions
under the assumptions of part (a) and (b), respectively.

Finally, we also have a Liouville type result for holomorphic forms in the
similar spirit of Theorem \ref{holo}.

\begin{theorem}
\label{form}Let $M^n$ be a complete K\"{a}hler manifold. Assume that there
exists a proper function $f$ on $M$ such that $f_{\alpha \beta }=0$ with
respect to unitary frames and $|\nabla f|$ is bounded. Then for all $0\leq
p\leq n,$ $\dim F(p)<\infty,$ where $F(p)$ denotes the space of holomorphic $%
(p,0)$ forms $\omega$ on $M$ with $\int_M |\omega|^2<\infty.$
\end{theorem}

\section{Proof of Theorem \protect\ref{holo}}

In this section, we give proof to Theorem \ref{holo}. Throughout, we assume $%
M$ is a complete K\"ahler manifold and $f$ a proper function on $M$ such
that $f_{\alpha \beta}=0$ with respect to unitary frames on $M.$ Without
loss of generality, we may assume $f$ is positive with minimum value $c_0.$
Let us denote
\begin{equation*}
D\left( t\right) :=\left\{ x\in M:f\left( x\right) \leq t\right\} .
\end{equation*}%
We assume everywhere that $c_0< t <\sup_M f$ so that $D\left( t\right) $ is
nonempty. Notice that $D\left( t\right) $ is compact for any such $t$ as $f$
is proper.

Let $u$ be a holomorphic function which is not identically zero on $M$. We
define a sequence of functions $\left( u_{k}\right) _{k\geq 0}$ as follows.
We set $u_{0}:=u$ and define inductively%
\begin{equation}
u_{k+1}:=\left\langle \nabla u_{k},\nabla f\right\rangle =\left(
u_{k}\right) _{\alpha }f_{\overline{\alpha }},\ \ \ \ \text{for }k\geq 0.
\label{a3}
\end{equation}

\begin{lemma}
\label{l1} $u_{k}$ is holomorphic for any $k\geq 0$.
\end{lemma}

\begin{proof}[Proof of Lemma \protect\ref{l1}]
We show this by induction on $k.$ For $k=0$ this is obviously true. Assuming
it is true for $k\geq 0,$ we prove that $u_{k+1}$ is holomorphic. Indeed,
with respect to unitary frames, one has
\begin{equation*}
\left( u_{k+1}\right) _{\overline{\delta }}=\left( \left( u_{k}\right)
_{\alpha }f_{\overline{\alpha }}\right) _{\overline{\delta }}=\left(
u_{k}\right) _{\alpha \overline{\delta }}f_{\overline{\alpha }}+\left(
u_{k}\right) _{\alpha }f_{\overline{\alpha }\overline{\delta }}=0,
\end{equation*}%
where the last equality holds true because $u_{k}$ is holomorphic by the
induction hypothesis and function $f$ satisfies $f_{\alpha \beta }=0.$
\end{proof}

\begin{lemma}
\label{l2} Let $u$ be a nonzero holomorphic function on $M.$ If

\begin{equation*}
\int_{D\left( t\right) }\left\vert u_{k}\right\vert ^{2}\leq c\,\mu ^{k}
\end{equation*}%
for all $k\geq 0,$ where $c$ and $\mu $ are constants independent of $k,$
then for any regular value $r$ of $f$ with $r\leq t$ we have
\begin{equation}
\int_{\partial D\left( r\right) }\left\vert u_{1}\right\vert ^{2}\left\vert
\nabla f\right\vert ^{-1}\leq \mu \int_{\partial D\left( r\right)
}\left\vert u_{0}\right\vert ^{2}\left\vert \nabla f\right\vert ^{-1}.
\label{ineq}
\end{equation}
\end{lemma}

\begin{proof}[Proof of Lemma \protect\ref{l2}]
For a regular value $s$ of $f,$ let us denote%
\begin{equation}
\rho \left( s\right) :=\frac{\int_{\partial D\left( s\right) }\left\vert
u_{1}\right\vert ^{2}\left\vert \nabla f\right\vert ^{-1}}{\int_{\partial
D\left( s\right) }\left\vert u_{0}\right\vert ^{2}\left\vert \nabla
f\right\vert ^{-1}}\geq 0.  \label{a4}
\end{equation}%
Notice that $\rho \left( s\right) <\infty.$ Otherwise, by the unique
continuation property of holomorphic functions, it implies $u_{0}=0$ on $M.$

We now apply the following well known formula
\begin{equation}
\int_{D\left( s\right) }\left( w\Delta v-v\Delta w\right) =\int_{\partial
D\left( s\right) }\left( w\frac{\partial v}{\partial \nu }-v\frac{\partial w%
}{\partial \nu }\right)  \label{a5}
\end{equation}%
to
\begin{equation*}
w:=u_{k}\text{ \ \ and \ }v:=\overline{u_{k+1}},
\end{equation*}%
where $\frac{\partial }{\partial \nu }=\frac{\nabla f}{\left\vert \nabla
f\right\vert }$ is the unit normal vector to $\partial D\left( s\right) .$

Observe that
\begin{eqnarray*}
\frac{\partial u_{k}}{\partial \nu } &=&\frac{1}{\left\vert \nabla
f\right\vert }\left\langle \nabla u_{k},\nabla f\right\rangle =\frac{1}{%
\left\vert \nabla f\right\vert }u_{k+1},\text{ \ and} \\
\frac{\partial \overline{u_{k+1}}}{\partial \nu } &=&\frac{1}{\left\vert
\nabla f\right\vert }\left\langle \nabla \overline{u_{k+1}},\nabla
f\right\rangle =\frac{1}{\left\vert \nabla f\right\vert }\overline{u_{k+2}}.
\end{eqnarray*}%
Also, notice that both $u_{k}$ and $\overline{u_{k+1}}$ are harmonic as both
are holomorphic. We deduce from (\ref{a5}) that%
\begin{equation}
\int_{\partial D\left( s\right) }\left( u_{k}\overline{u_{k+2}}-\left\vert
u_{k+1}\right\vert ^{2}\right) \left\vert \nabla f\right\vert ^{-1}=0.
\label{a6}
\end{equation}%
By the Cauchy-Schwarz inequality, we obtain

\begin{gather*}
\left( \int_{\partial D\left( s\right) }\left\vert u_{k+1}\right\vert
^{2}\left\vert \nabla f\right\vert ^{-1}\right) ^{2}\leq \left(
\int_{\partial D\left( s\right) }\left\vert u_{k}\right\vert \left\vert
u_{k+2}\right\vert \left\vert \nabla f\right\vert ^{-1}\right) ^{2} \\
\leq \left( \int_{\partial D\left( s\right) }\left\vert u_{k}\right\vert
^{2}\left\vert \nabla f\right\vert ^{-1}\right) \left( \int_{\partial
D\left( s\right) }\left\vert u_{k+2}\right\vert ^{2}\left\vert \nabla
f\right\vert ^{-1}\right) .
\end{gather*}%
We have thus proved that
\begin{gather}
\left( \int_{\partial D\left( s\right) }\left\vert u_{k+1}\right\vert
^{2}\left\vert \nabla f\right\vert ^{-1}\right) ^{2}\leq \left(
\int_{\partial D\left( s\right) }\left\vert u_{k}\right\vert ^{2}\left\vert
\nabla f\right\vert ^{-1}\right)  \label{a7} \\
\times \left( \int_{\partial D\left( s\right) }\left\vert u_{k+2}\right\vert
^{2}\left\vert \nabla f\right\vert ^{-1}\right)  \notag
\end{gather}%
for any $k\geq 0.$

Multiplying the inequality (\ref{a7}) from $k=0$ to $l,$ we conclude

\begin{eqnarray*}
&&\left( \int_{\partial D\left( s\right) }\left\vert u_{1}\right\vert
^{2}\left\vert \nabla f\right\vert ^{-1}\right) \left( \int_{\partial
D\left( s\right) }\left\vert u_{l}\right\vert ^{2}\left\vert \nabla
f\right\vert ^{-1}\right) \\
&\leq &\left( \int_{\partial D\left( s\right) }\left\vert u_{0}\right\vert
^{2}\left\vert \nabla f\right\vert ^{-1}\right) \left( \int_{\partial
D\left( s\right) }\left\vert u_{l+1}\right\vert ^{2}\left\vert \nabla
f\right\vert ^{-1}\right) .
\end{eqnarray*}%
In view of the definition of $\rho (s) $ in (\ref{a4}), this means that
\begin{equation}
\int_{\partial D\left( s\right) }\left\vert u_{l+1}\right\vert
^{2}\left\vert \nabla f\right\vert ^{-1}\geq \rho \left( s\right)
\int_{\partial D\left( s\right) }\left\vert u_{l}\right\vert ^{2}\left\vert
\nabla f\right\vert ^{-1}.  \label{a8}
\end{equation}%
Iterating (\ref{a8}) from $l=0$ to $k-1,$ we arrive at
\begin{equation*}
\int_{\partial D\left( s\right) }\left\vert u_{k}\right\vert ^{2}\left\vert
\nabla f\right\vert ^{-1}\geq \rho ^{k} \left( s\right) \,\int_{\partial
D\left( s\right) }\left\vert u_{0}\right\vert ^{2}\left\vert \nabla
f\right\vert ^{-1} \text{ \ \ for all }k\geq 0.
\end{equation*}

We now integrate the preceding inequality from $r-\delta $ to $r,$ where $%
\delta >0$ is small so that any $s\in \left[ r-\delta ,r\right] $ is a
regular value for $f $. The co-area formula implies that
\begin{eqnarray*}
&&\int_{D\left( r\right) \backslash D\left( r-\delta \right) }\left\vert
u_{k}\right\vert ^{2} \\
&\geq &\int_{r-\delta }^{r}\rho ^{k}\left( s\right) \left( \int_{\partial
D\left( s\right) }\left\vert u_{0}\right\vert ^{2}\left\vert \nabla
f\right\vert ^{-1}\right) ds \\
&\geq &\left( \inf_{r-\delta \leq s\leq r}\int_{\partial D\left( s\right)
}\left\vert u_{0}\right\vert ^{2}\left\vert \nabla f\right\vert ^{-1}\right)
\int_{r-\delta }^{r}\rho ^{k}\left( s\right) ds.
\end{eqnarray*}%
Clearly,
\begin{equation*}
\inf_{r-\delta \leq s\leq r}\int_{\partial D\left( s\right) }\left\vert
u_{0}\right\vert ^{2}\left\vert \nabla f\right\vert ^{-1}>0
\end{equation*}%
as otherwise, it would mean $u_{0}=0$ on a level set of $f,$ which in turn
implies $u_{0}=0$ on $M.$

Thus, we have proved that

\begin{eqnarray*}
\int_{r-\delta }^{r}\rho ^{k}\left( s\right) ds &\leq &C_{1}\,\int_{D\left(
r\right) \backslash D\left( r-\delta \right) }\left\vert u_{k}\right\vert
^{2} \\
&\leq &C_{1}\,\int_{D\left( t\right) }\left\vert u_{k}\right\vert ^{2} \\
&\leq &C_{1}\,c\,\mu ^{k},
\end{eqnarray*}%
where $C_{1}$ is independent of $k$ and $c$ is the constant in the
hypothesis.

Rewriting the inequality into
\begin{equation*}
\int_{r-\delta }^{r}\left( \frac{\rho \left( s\right) }{\mu }\right)
^{k}ds\leq C_{2}
\end{equation*}%
and letting $k\rightarrow \infty ,$ one concludes that
\begin{equation*}
\rho \left( r\right) \leq \mu .
\end{equation*}%
Hence,
\begin{equation}
\int_{\partial D\left( r\right) }\left\vert u_{1}\right\vert ^{2}\left\vert
\nabla f\right\vert ^{-1}\leq \mu \,\int_{\partial D\left( r\right)
}\left\vert u_{0}\right\vert ^{2}\left\vert \nabla f\right\vert ^{-1}.
\label{a12}
\end{equation}%
This proves the lemma.
\end{proof}

In view of Lemma \ref{l2}, it is important to control the growth rate of $%
\left\vert u_{k}\right\vert .$ This is done in the following lemma.

\begin{lemma}
\label{l3}

(a) If $|\nabla f|\leq b$ and $V_p(R)\le C\,e^{a\,R}$ for all $R\ge 1,$
where $V_p(R)$ denotes the volume of the geodesic ball $B_p(R)$ in $M,$ then
for any holomorphic function $u$ satisfying

\begin{equation*}
|u|(x)\leq c\,e^{d\,r(x)},
\end{equation*}%
there exists a constant $C_{0}$ independent of $k$ such that

\begin{equation*}
\int_{B_{p}\left( r\right) }\left\vert u_{k}\right\vert ^{2}\leq
C_{0}\,\left( e^{(2d+a)r}\right) ^{k}
\end{equation*}%
for any $r\geq 2b.$

(b) If $|\nabla f|(x)\le b\,\left(r(x)+1\right)$ and $V_p(R)\le C\,R^m$ for
all $R\ge 1,$ then for any holomorphic function $u$ satisfying

\begin{equation*}
|u|(x)\le c\,(r(x)+1)^d,
\end{equation*}
we have

\begin{equation*}
\int_{B_{p}\left( r\right) }\left\vert u_{k}\right\vert ^{2}\leq
C\,_{0}\left( (b+1)^{2}\,2^{m+2d+2}\right) ^{k}
\end{equation*}%
for a constant $C_{0}>0$ independent of $k$ and any $r\geq b.$
\end{lemma}

\begin{proof}[Proof of Lemma \protect\ref{l3}]
We prove part (a) first. Consider the cut-off function%
\begin{equation*}
\phi \left( x\right) :=\left\{
\begin{array}{c}
1 \\
\frac{1}{r}\left( t+r-d\left( p,x\right) \right) \\
0%
\end{array}%
\right.
\begin{array}{c}
\text{on }B_{p}\left( t\right) \\
\text{on }B_{p}\left( t+r\right) \backslash B_{p}\left( t\right) \\
\text{on }M\backslash B_{p}\left( t+r\right) .%
\end{array}%
\end{equation*}%
Integrating by parts, we get%
\begin{eqnarray*}
\int_{M}\left\vert \nabla u_{k}\right\vert ^{2}\phi ^{2} &=&-\int_{M}\left(
\Delta u_{k}\right) \overline{u_{k}}\phi ^{2}-2\int_{M}\left\langle \nabla
u_{k},\nabla \phi \right\rangle \phi \overline{u_{k}} \\
&\leq &2\int_{M}\left\vert \nabla u_{k}\right\vert \left\vert \nabla \phi
\right\vert \phi \left\vert u_{k}\right\vert \\
&\leq &\frac{1}{2}\int_{M}\left\vert \nabla u_{k}\right\vert ^{2}\phi
^{2}+2\int_{M}\left\vert u_{k}\right\vert ^{2}\left\vert \nabla \phi
\right\vert ^{2}.
\end{eqnarray*}%
Consequently, this proves that
\begin{eqnarray*}
\int_{B_{p}\left( t\right) }\left\vert \nabla u_{k}\right\vert ^{2} &\leq
&\int_{M}\left\vert \nabla u_{k}\right\vert ^{2}\phi ^{2} \\
&\leq &4\int_{M}\left\vert u_{k}\right\vert ^{2}\left\vert \nabla \phi
\right\vert ^{2} \\
&=&\frac{4}{r^{2}}\int_{B_{p}\left( t+r\right) }\left\vert u_{k}\right\vert
^{2}.
\end{eqnarray*}%
However,
\begin{eqnarray*}
\left\vert u_{k+1}\right\vert &=&\left\vert \left\langle \nabla u_{k},\nabla
f\right\rangle \right\vert \\
&\leq &\left\vert \nabla u_{k}\right\vert \left\vert \nabla f\right\vert \\
&\leq & b\,\left\vert \nabla u_{k}\right\vert .
\end{eqnarray*}%
Combining this with the above estimate, we conclude
\begin{equation}
\int_{B_{p}\left( t\right) }\left\vert u_{k+1}\right\vert ^{2}\leq \frac{%
4\,b^2}{r^{2}}\int_{B_{p}\left( t+r\right) }\left\vert u_{k}\right\vert ^{2}.
\label{a9}
\end{equation}

This is true for any $t>0$ and for any $k\geq 0$. Iterating (\ref{a9}) leads
to
\begin{eqnarray*}
\int_{B_{p}\left( t\right) }\left\vert u_{k+1}\right\vert ^{2} &\leq &\frac{%
4\,b^{2}}{r^{2}}\int_{B_{p}\left( t+r\right) }\left\vert u_{k}\right\vert
^{2} \\
&\leq &\left( \frac{4\,b^{2}}{r^{2}}\right) ^{2}\int_{B_{p}\left(
t+2r\right) }\left\vert u_{k-1}\right\vert ^{2} \\
&\leq &\left( \frac{4\,b^{2}}{r^{2}}\right) ^{k+1}\int_{B_{p}\left( t+\left(
k+1\right) r\right) }\left\vert u_{0}\right\vert ^{2}.
\end{eqnarray*}%
Choosing now $t=r$ in the above inequality implies%
\begin{equation}
\int_{B_{p}\left( r\right) }\left\vert u_{k}\right\vert ^{2}\leq \left(
\frac{4\,b^{2}}{r^{2}}\right) ^{k}\int_{B_{p}\left( \left( k+1\right)
r\right) }\left\vert u_{0}\right\vert ^{2}.  \label{a10}
\end{equation}%
By the volume growth assumption, we know that
\begin{equation*}
V_{p}\left( (k+1)r)\right) \leq C\,e^{a(k+1)r}.
\end{equation*}%
In view of the growth assumption on $u,$ we conclude
\begin{eqnarray*}
\int_{B_{p}\left( r\right) }\left\vert u_{k}\right\vert ^{2} &\leq &\left(
\frac{4\,b^{2}}{r^{2}}\right) ^{k}\left( \sup_{B_{p}\left( \left( k+1\right)
r\right) }\left\vert u_{0}\right\vert ^{2}\right) \,V_{p}\left(
(k+1)r)\right) \\
&\leq &c^{2}C\,\left( \frac{4\,b^{2}}{r^{2}}\right) ^{k}\,e^{(2d+a)(k+1)r}.
\end{eqnarray*}%
So the conclusion of part (a) follows by noticing that $r\geq 2b.$

To prove part (b), we start from the inequality

\begin{equation*}
\int_{B_{p}\left( r\right) }\left\vert \nabla u_{k}\right\vert ^{2} \leq
\frac{4}{r^{2}}\int_{B_{p}(2r)}\left\vert u_{k}\right\vert ^{2}.
\end{equation*}%
Since $\left\vert\nabla f\right\vert \left( x\right) \leq b\,\left(r(x)
+1\right) $ for all $x\in M,$ it follows that
\begin{eqnarray*}
\left\vert u_{k+1}\right\vert \left( x\right) &=&\left\vert \left\langle
\nabla u_{k},\nabla f\right\rangle \right\vert \left( x\right) \\
&\leq & (b+1)\,r\left( x\right) \left\vert \nabla u_{k}\right\vert \left(
x\right)
\end{eqnarray*}%
for $r\left( x\right) \geq b.$

Combining this with the above estimate, for $r\geq b,$ we have
\begin{equation}
\int_{B_{p}\left( r\right) }\left\vert u_{k+1}\right\vert ^{2}\leq
4\,(b+1)^{2}\,\int_{B_{p}\left( 2r\right) }\left\vert u_{k}\right\vert ^{2}.
\label{s4}
\end{equation}%
This is true for any $k\geq 0.$ By iterating (\ref{s4}) it follows that
\begin{equation}
\int_{B_{p}\left( r\right) }\left\vert u_{k}\right\vert ^{2}\leq \left(
4(b+1)^{2}\right) ^{k}\,\int_{B_{p}\left( 2^{k}r\right) }\left\vert
u_{0}\right\vert ^{2}.  \label{s5}
\end{equation}%
Since
\begin{equation*}
V_{p}\left( R\right) \leq C\,R^{m}\text{ \ for all }R\geq 1
\end{equation*}%
and
\begin{equation*}
|u|(x)\leq c\,(r(x)+1)^{d},
\end{equation*}%
it follows that
\begin{eqnarray*}
\int_{B_{p}\left( r\right) }\left\vert u_{k}\right\vert ^{2} &\leq &\left(
4(b+1)^{2}\right) ^{k}\,\left( \sup_{B_{p}\left( 2^{k}r\right) }\left\vert
u_{0}\right\vert ^{2}\right) \,V_{p}\left( 2^{k}r\right) \\
&\leq &c^{2}C\,\left( 4(b+1)^{2}\,2^{(m+2d)}\right) ^{k}\,r^{m+2d}.
\end{eqnarray*}%
We have thus proved that
\begin{equation}
\int_{B_{p}\left( r\right) }\left\vert u_{k}\right\vert ^{2}\leq C\,_{0}\mu
^{k}  \label{s6}
\end{equation}%
for all $k\geq 0$ and $r\geq b,$ where $\mu :=(b+1)^{2}\,2^{m+2d+2}.$ This
proves the lemma.
\end{proof}

We are now in position to prove Theorem \ref{holo}. For the sake of
convenience, we state it again here.

\begin{theorem}
\label{holo'}Let $M$ be a complete K\"{a}hler manifold. Assume that there
exists a proper function $f$ on $M$ such that $f_{\alpha \beta }=0$ with
respect to unitary frames. Then for all $d>0,$

(a) $\dim E(d)<\infty$ if $|\nabla f|$ is bounded and $M$ has exponential
volume growth.

(b) $\dim P(d)<\infty$ if $|\nabla f|$ grows at most linearly and $M$ has
polynomial volume growth.
\end{theorem}

\begin{proof}[Proof of Theorem \protect\ref{holo'}]
Fix a regular value $t_{0}$ of $f$ and choose $r_{0}>0$ so that $D\left(
t_{0}\right) \subset B_{p}\left( r_{0}\right) $. By Lemma \ref{l3}, for $%
u\in E(d)$ under the assumptions of (a) or for $u\in P(d)$ under the
assumptions of (b), we have
\begin{equation*}
\int_{D\left( t_{0}\right) }\left\vert u_{k}\right\vert ^{2}\leq
\int_{B_{p}\left( r_{0}\right) }\left\vert u_{k}\right\vert ^{2}\leq
C_{0}\mu ^{k}
\end{equation*}%
for some constant $\mu $ independent of $k.$ So by Lemma \ref{l2},
\begin{equation*}
\int_{\partial D\left( t\right) }\left\vert u_{1}\right\vert ^{2}\left\vert
\nabla f\right\vert ^{-1}\leq \mu \int_{\partial D\left( t\right)
}\left\vert u_{0}\right\vert ^{2}\left\vert \nabla f\right\vert ^{-1}
\end{equation*}%
for any regular value $t$ of $f$ such that $c_{0}\leq t\leq t_{0}.$ Since $%
u_{1}$ is zero whenever $\nabla f=0,$ it is easy to see, by the co-area
formula, that%
\begin{equation*}
\int_{D\left( t_{0}\right) }\left\vert u_{1}\right\vert ^{2}\leq \mu
\int_{D\left( t_{0}\right) }\left\vert u_{0}\right\vert ^{2}.
\end{equation*}

Since
\begin{eqnarray*}
\int_{D\left( t_{0}\right) }\left\langle \nabla \left\vert u_{0}\right\vert
^{2},\nabla f\right\rangle &=&\int_{D\left( t_{0}\right) }u_{0}\overline{%
u_{1}}+\overline{u_{0}}u_{1} \\
&\leq &2\int_{D\left( t_{0}\right) }\left\vert u_{0}\right\vert \left\vert
u_{1}\right\vert \\
&\leq &\int_{D\left( t_{0}\right) }\left\vert u_{0}\right\vert
^{2}+\int_{D\left( t_{0}\right) }\left\vert u_{1}\right\vert ^{2},
\end{eqnarray*}%
we conclude

\begin{equation}
\int_{D\left( t_{0}\right) }\left\langle \nabla \left\vert u_{0}\right\vert
^{2},\nabla f\right\rangle \leq \left( 1+\mu\right) \int_{D\left(
t_{0}\right) }\left\vert u_{0}\right\vert ^{2}.  \label{a14'}
\end{equation}%
Plugging this into the following integration by parts formula
\begin{equation}
\int_{D\left( t_{0}\right) }\left\langle \nabla \left\vert u_{0}\right\vert
^{2},\nabla f\right\rangle =-\int_{D\left( t_{0}\right) }\left\vert
u_{0}\right\vert ^{2}\Delta f+\int_{\partial D\left( t_{0}\right)
}\left\vert u_{0}\right\vert ^{2}\frac{\partial f}{\partial \nu }
\label{a14}
\end{equation}%
and noting that $\frac{\partial f}{\partial \nu }=\left\vert \nabla
f\right\vert,$ we arrive at
\begin{equation}
\int_{\partial D\left( t_{0}\right) }\left\vert u_{0}\right\vert
^{2}\left\vert \nabla f\right\vert \leq \left( \sup_{D\left( t_{0}\right)
}\left\vert \Delta f\right\vert +1+\mu\right) \int_{D\left( t_{0}\right)
}\left\vert u_{0}\right\vert ^{2}.  \label{a15}
\end{equation}%
Since $t_0$ is a regular value of $f,$
\begin{equation*}
C_1\left( t_{0}\right) :=\inf_{\partial D\left( t_{0}\right) }\left\vert
\nabla f\right\vert >0.
\end{equation*}%
Therefore, we may rewrite (\ref{a15}) into

\begin{equation*}
\int_{\partial D\left( t_{0}\right) }\left\vert u_{0}\right\vert
^{2}\left\vert \nabla f\right\vert ^{-1}\leq C_{2}\left( t_{0}\right)
\int_{D\left( t_{0}\right) }\left\vert u_{0}\right\vert ^{2},  
\end{equation*}%
where $C_2(t_0)$ is a constant depending on $\mu$, $\sup_{D(t_0)}\,|\Delta f|$ and $%
C_1(t_0).$

For the regular value $t_{0}$ of $f,$ we let $\varepsilon >0$ be
sufficiently small so that any $t$ with $t_{0}-\varepsilon \leq t\leq t_{0}$
is a regular value of $f$ as well. Such $\varepsilon $ depends on $%
\sup_{D\left( t_{0}\right) }\left\vert \mathrm{Hess}\left( f\right)
\right\vert $. The preceding argument also implies

\begin{equation}
\int_{\partial D\left( t\right) }\left\vert u_{0}\right\vert ^{2}\left\vert
\nabla f\right\vert ^{-1}\leq C\left( t_{0}\right) \int_{D\left( t\right)
}\left\vert u_{0}\right\vert ^{2} \label{a16}
\end{equation}%
with the constant $C(t_{0})$ now depending on $\mu$, $\sup_{D(t_{0})}\,|\Delta f|$
and $C_{3}(t_{0}),$ where

\begin{equation*}
C_{3}\left( t_{0}\right) :=\inf_{D\left( t_{0}\right) \backslash D\left(
t_{0}-\varepsilon \right) }\left\vert \nabla f\right\vert >0.
\end{equation*}

Integrating (\ref{a16}) from $t:=t_{0}-\varepsilon $ to $t:=t_{0}$ implies
\begin{equation}
\int_{D\left( t_{0}\right) }\left\vert u_{0}\right\vert ^{2}\leq
e^{\varepsilon \,C\left( t_{0}\right) }\int_{D\left( t_{0}-\varepsilon
\right) }\left\vert u_{0}\right\vert ^{2}.  \label{a18}
\end{equation}%
The inequality (\ref{a18}) is true for any $u=u_{0}\in E(d)$ in case of (a)
and $u=u_{0}\in P(d)$ in case of (b). It is well known that this implies a
dimension estimate as claimed in the theorem. We will follow \cite{L2} to
supply some details here. Denote by $H$ to be $E(d)$ in case of (a) and $%
P(d) $ in case of (b). By a result of P. Li (see \cite{L2}), there exists $%
u_{0}\in H$ so that
\begin{equation}
\int_{D(t_0-\varepsilon)}\left\vert u_{0}\right\vert ^{2}\leq \frac{%
n\,V(D(t_0-\varepsilon))} {\dim H}\sup_{D(t_0-\varepsilon)}\left\vert
u_{0}\right\vert ^{2}.
\end{equation}%
On the other hand, applying the Moser iteration argument to the subharmonic
function $|u_0|$, one obtains that
\begin{equation}
\sup_{ D(t_0-\varepsilon)}\left\vert u_{0}\right\vert ^{2}\leq \frac{C}{%
V(D(t_0))}\, \int_{D(t_0)}\left\vert u_{0}\right\vert ^{2}  \label{a20}
\end{equation}%
for some constant $C$ depending on $D(t_0).$ By combining these two
inequalities, it follows that
\begin{equation*}
\int_{D(t_0-\varepsilon)}\left\vert u_{0}\right\vert ^{2}\leq \frac{C}{\dim H%
}\int_{D(t_0)}\left\vert u_{0}\right\vert ^{2}.
\end{equation*}%
In view of (\ref{a18}), one concludes $\dim H\leq C.$ The theorem is proved.
\end{proof}

\section{Shrinking solitons}

In this section we prove Theorem \ref{Shrinking} which is restated below.

\begin{theorem}
\label{shrink}Let $\left( M^{n},g,f\right) $ be a gradient K\"{a}hler Ricci
shrinking soliton of complex dimension $n$. Then $\dim P(d)\leq C(n,d),$ a
constant depending only on $n$ and $d,$ for all $d>0.$
\end{theorem}

\begin{proof}[Proof of Theorem \protect\ref{shrink}]
Recall that on a shrinking soliton, by normalizing the metric so that $%
\lambda =\frac{1}{2}$ and adding a suitable constant to the potential
function $f,$ one has

\begin{equation}
S+\left\vert \nabla f\right\vert ^{2}=f \text{ and \ \ }S+\Delta f=\frac{n}{2%
}.  \label{s1}
\end{equation}%
Note that by Chen \cite{C}
\begin{equation*}
S\geq 0.
\end{equation*}
So we have
\begin{equation*}
|\nabla f|^2\leq f.
\end{equation*}
Also, by Cao and Zhou \cite{CZ},
\begin{equation}
\frac{1}{4}\left( d\left( p,x\right) -c\left( n\right) \right) ^{2}\leq
f\left( x\right) \leq \frac{1}{4}\left( d\left( p,x\right) +c\left( n\right)
\right) ^{2}  \label{s2}
\end{equation}%
and

\begin{equation*}
V_p(R)\le c(n)\,R^n
\end{equation*}
for $R\geq 1,$ where $p\in M$ is a minimum point for $f$ and the constant $%
c(n)$ depends only on the dimension $n$ of $M.$

In particular, one concludes that $f$ is proper on $M$ and
\begin{equation*}
|\nabla f|(x)\le \frac{1}{2}\,r(x)+c(n).
\end{equation*}
So by Lemma \ref{l3}, for $u\in P(d)$ and $r\geq c(n),$
\begin{equation}
\int_{B_{p}\left( r\right) }\left\vert u_{k}\right\vert ^{2}\leq C\, \mu ^{k}
\end{equation}%
for all $k\geq 0,$ where $\mu=c(n,d).$

Now by Lemma \ref{l2} we have
\begin{equation}
\int_{\partial D\left( s\right) }\left\vert u_{1}\right\vert ^{2}\left\vert
\nabla f\right\vert ^{-1}\leq \mu \int_{\partial D\left( s\right)
}\left\vert u_{0}\right\vert ^{2}\left\vert \nabla f\right\vert ^{-1},
\label{s8}
\end{equation}%
for any regular value $s$ of $f.$ Using (\ref{s8}) and the co-area formula,
we obtain
\begin{eqnarray*}
\int_{D\left( t\right) }\left\vert u_{1}\right\vert ^{2}e^{-f}
&=&\int_{0}^{t}e^{-s}\left( \int_{\partial D\left( s\right) }\left\vert
u_{1}\right\vert ^{2}\left\vert \nabla f\right\vert ^{-1}\right) ds \\
&\leq &\mu \int_{0}^{t}e^{-s}\left( \int_{\partial D\left( s\right)
}\left\vert u_{0}\right\vert ^{2}\left\vert \nabla f\right\vert ^{-1}\right)
ds \\
&=&\mu \int_{D\left( t\right) }\left\vert u_{0}\right\vert ^{2}e^{-f}.
\end{eqnarray*}%
Hence, we have established that
\begin{equation}
\int_{D\left( t\right) }\left\vert u_{1}\right\vert ^{2}e^{-f}\leq \mu
\int_{D\left( t\right) }\left\vert u_{0}\right\vert ^{2}e^{-f}  \label{s9}
\end{equation}%
for all $t.$ Using (\ref{s1}) we have that
\begin{gather}
\int_{D\left( t\right) }\left\vert u_{0}\right\vert ^{2}\left( f-\frac{n}{2}%
\right) e^{-f}=-\int_{D\left( t\right) }\left\vert u_{0}\right\vert
^{2}\Delta _{f}\left( f\right) e^{-f}  \label{s10} \\
=\int_{D\left( t\right) }\left\langle \nabla \left\vert u_{0}\right\vert
^{2},\nabla f\right\rangle e^{-f}-\int_{\partial D\left( t\right)
}\left\vert u_{0}\right\vert ^{2}\frac{\partial f}{\partial \nu }e^{-f}
\notag \\
\leq \int_{D\left( t\right) }\left\langle \nabla \left\vert u_{0}\right\vert
^{2},\nabla f\right\rangle e^{-f}.  \notag
\end{gather}%
In the last line above, we have used that $\frac{\partial f}{\partial \nu }%
=\left\vert \nabla f\right\vert \geq 0.$ As in the proof of (\ref{a14'}), we
have
\begin{eqnarray*}
\left\vert \int_{D\left( t\right) }\left\langle \nabla \left\vert
u_{0}\right\vert ^{2},\nabla f\right\rangle e^{-f}\right\vert &\leq
&2\int_{D\left( t_{0}\right) }\left\vert u_{0}\right\vert \left\vert
u_{1}\right\vert e^{-f} \\
&\leq &\int_{D\left( t\right) }\left\vert u_{0}\right\vert
^{2}e^{-f}+\int_{D\left( t\right) }\left\vert u_{1}\right\vert ^{2}e^{-f} \\
&\leq &\left( 1+\mu \right) \int_{D\left( t\right) }\left\vert
u_{0}\right\vert ^{2}e^{-f},
\end{eqnarray*}
where in the last step we have used (\ref{s9}). By (\ref{s10}), this means
that
\begin{equation*}
\int_{D\left( t\right) }\left\vert u_{0}\right\vert ^{2}\left( f-\frac{n}{2}%
-1-\mu \right) e^{-f}\leq 0 \text{ \ for all }t.
\end{equation*}%
Since the constant $\mu=c(n,d),$ by choosing sufficiently large $t=t_{0}$
depending only on $n$ and $d$ it follows that
\begin{equation}
\int_{D\left( 5\,t_{0}\right) }\left\vert u_{0}\right\vert ^{2}\leq
K\int_{D\left( t_{0}\right) }\left\vert u_{0}\right\vert ^{2},  \label{s11}
\end{equation}%
where $K$ is a constant depending only on $n$ and $d.$ In view of (\ref{s2}%
), $t_0$ can be chosen in such a way that

\begin{equation*}
D\left( t_{0}\right)\subset B_p(3\,r_0)
\end{equation*}
and
\begin{equation*}
B_p(4\,r_0)\subset D\left( 5\,t_{0}\right),
\end{equation*}
where $r_0=\sqrt{t_0}.$ Hence, we conclude

\begin{equation}
\int_{B_p \left( 4\,r_{0}\right) }\left\vert u_{0}\right\vert ^{2}\leq
K\,\int_{B_p \left( 3\,r_{0}\right) }\left\vert u_{0}\right\vert ^{2}
\label{s12}
\end{equation}%
for some $r_0$ depending only on $n$ and $d.$

Notice that (\ref{s12}) is true for any $u_{0}\in P(d).$

This implies that $\dim P(d) \leq C\left( n,d\right) .$ Indeed, by a result
of P. Li (see \cite{L2}), there exists nontrivial $u_{0}\in P(d)$ so that
\begin{equation}
\int_{B_p\left(3\, r_{0}\right) }\left\vert u_{0}\right\vert ^{2}\leq \frac{%
n\,V_p(3r_0)} {\dim P(d)}\,\sup_{B_p \left(3\, r_{0}\right) }\left\vert
u_{0}\right\vert ^{2}.  \label{s13}
\end{equation}%
On the other hand, the Sobolev constant of $B_p\left( 4\,r_{0}\right) $ can
be controlled by a constant depending only on $n$ and $r_0$ (see \cite{MW}).
Therefore, applying the Moser iteration argument to the subharmonic function
$|u_0|$, we obtain that
\begin{equation}
\sup_{B_p \left(3\,r_{0}\right) }\left\vert u_{0}\right\vert ^{2}\leq \frac{%
C( n,d) }{V_p(4\,r_{0})}\, \int_{B_p\left( 4\,r_{0}\right) }\left\vert
u_{0}\right\vert ^{2}.  \label{s14}
\end{equation}%
By combining (\ref{s13}) and (\ref{s14}), it follows that
\begin{equation*}
\int_{B_p \left(3\,r_{0}\right) }\left\vert u_{0}\right\vert ^{2}\leq \frac{%
C(n,d)}{\dim P(d)}\int_{B_p \left( 4\,r_{0}\right) }\left\vert
u_{0}\right\vert ^{2}.
\end{equation*}%
In view of (\ref{s12}), we have $\dim P(d)\leq C(n,d).$ This proves the
theorem.
\end{proof}

\section{Holomorphic forms}

In this section, we will deal with the space of holomorphic forms and prove
Theorem \ref{form} which is restated below.

\begin{theorem}
\label{forms}Let $\left( M^{n},g\right) $ be a complete K\"{a}hler manifold.
Assume that there exists a proper smooth function $f$ on $M$ with bounded
gradient and $f_{\alpha \beta }=0$ in unitary frames. Then the dimension of
the space of $L^{2}$ holomorphic $\left( p,0\right) $ forms is finite for $%
0\leq p\leq n.$
\end{theorem}

\begin{proof}[Proof of Theorem \protect\ref{forms}]
We use induction on $p.$ For $p=0,$ the statement is clear as any $L^{2}$
holomorphic function must be a constant (see \cite{Y4}).

We now assume that the result is true for all $0\leq p\leq q-1$ and prove it
for $p=q.$ Let $F(p)$ denote the vector space of holomorphic $\left(
p,0\right) $ forms in $L^{2}\left( M\right).$ For any $\left( q,0\right) $
form
\begin{equation*}
\omega :=\frac{1}{q!}\omega _{i_{1}...i_{q}}dz^{i_{1}}\wedge ..\wedge
dz^{i_{q}},
\end{equation*}%
we associate a $\left( q-1,0\right) $ form $\theta ^{\omega }$ by
contracting it with $\nabla f.$ So
\begin{eqnarray*}
\theta ^{\omega } &:&=\omega \left( \cdot ,.,\cdot ,\nabla f\right) \\
\theta ^{\omega } &=&\frac{1}{\left( q-1\right) !}\left( \omega
_{i_{1}...i_{q}}f_{\overline{i_{q}}}\right) dz^{i_{1}}\wedge ..\wedge
dz^{i_{q-1}}.
\end{eqnarray*}%
We notice that $\theta ^{\omega }$ is a holomorphic form. Indeed,
\begin{equation*}
\nabla _{\bar{\alpha}}\theta _{i_{1}...i_{q-1}}^{\omega }=\nabla _{\bar{%
\alpha}}\left( \omega _{i_{1}...i_{q}}f_{\overline{i_{q}}}\right) =0
\end{equation*}%
by using that $\omega $ is holomorphic and $f_{\alpha \beta }=0$. Moreover, $%
\theta ^{\omega }$ is in $L^{2}(M)$ as $\left\vert \nabla f\right\vert $ is
bounded on $M.$ By the induction hypothesis, we know that the vector space
\begin{equation*}
\left\{ \theta ^{\omega }:\omega \text{ is }L^{2}\text{ holomorphic }\left(
p,0\right) \text{ form }\right\} \subset F(q-1)
\end{equation*}%
is finite dimensional. Therefore, to finish the proof it suffices to show
that the space
\begin{equation*}
F:=\left\{ \omega \in F(q) :\omega \left( \cdot ,.,\cdot ,\nabla f\right)
=0\right\} \subset F(q)
\end{equation*}%
is finite dimensional as well. For this, we consider $\omega \in F.$ First,
observe that $\omega $ is closed as it is harmonic and in $L^{2}.$ It
follows that
\begin{align}
\left\langle \nabla \left\vert \omega \right\vert ^{2},\nabla f\right\rangle
& =\frac{1}{q!}\left\langle \nabla \left\vert \omega
_{i_{1}...i_{q}}\right\vert ^{2},\nabla f\right\rangle  \label{h2} \\
& =\frac{2}{q!} Re\left\{ \left( \nabla _{\alpha }\omega
_{i_{1}...i_{q}}\right) f_{\bar{\alpha}}\bar{\omega}_{\overline{i_{1}}...%
\overline{i_{q}}}\right\}  \notag \\
& =\frac{2}{q!}Re\left\{ \left( \sum_{k}\varepsilon \left( k\right) \nabla
_{i_{k}}\omega _{i_{1}..\alpha ...i_{q}}\right) f_{\bar{\alpha}}\bar{\omega}%
_{\overline{i_{1}}...\overline{i_{q}}}\right\}  \notag
\end{align}%
for some $\varepsilon\left( k\right) \in \left\{ -1,1\right\} .$ On the
other hand, note that
\begin{eqnarray*}
\left( \nabla _{i_{k}}\omega _{i_{1}..\alpha ...i_{q}}\right) f_{\bar{\alpha}%
} &=&\nabla _{i_{k}}\left( \omega _{i_{1}..\alpha ...i_{q}}f_{\bar{\alpha}%
}\right) -\omega _{i_{1}..\alpha ...i_{q}}f_{i_{k}\bar{\alpha}} \\
&=&-\omega _{i_{1}..\alpha ...i_{q}}f_{i_{k}\bar{\alpha}}.
\end{eqnarray*}%
Therefore, from (\ref{h2}), one concludes
\begin{equation*}
\left\vert \left\langle \nabla \left\vert \omega \right\vert ^{2},\nabla
f\right\rangle \right\vert \leq C\left\vert \omega \right\vert ^{2}
\end{equation*}%
for some constant $C$ depending on the Hessian of $f$ on $D\left( r\right).$

Plugging this into the following equation
\begin{equation}
\int_{D\left( r\right) }\left\langle \nabla \left\vert \omega \right\vert
^{2},\nabla f\right\rangle =-\int_{D\left( r\right) }\left\vert \omega
\right\vert ^{2}\Delta f+\int_{\partial D\left( r\right) }\left\vert \omega
\right\vert ^{2}\left\vert \nabla f\right\vert ,  \label{h1}
\end{equation}%
where $r$ is a regular value of $f,$ we have

\begin{equation}
\int_{\partial D\left( r\right) }\left\vert \omega \right\vert
^{2}\left\vert \nabla f\right\vert ^{-1}\leq C\int_{D\left( r\right)
}\left\vert \omega \right\vert ^{2}  \label{h3}
\end{equation}%
for a constant $C$ depending on the Hessian of $f$ on $D\left( r\right) $
and on a lower bound of $\left\vert \nabla f\right\vert $ on the set $%
\partial D\left( r\right) =\left\{ f=r\right\} .$

For a fixed regular value $r_{0}$ of $f,$ choose $\varepsilon>0$ so that for
$r$ is also a regular value for $r_0-\varepsilon\leq r\leq r_0.$ Integrating
(\ref{h3}) from $r_{0}-\varepsilon $ to $r_{0},$ we obtain that

\begin{equation}
\int_{D\left( r_{0}\right) }\left\vert \omega \right\vert ^{2}\leq
C\,\int_{D\left( r_{0}-\varepsilon \right) }\left\vert \omega \right\vert
^{2}  \label{h4}
\end{equation}
for any $\omega \in F.$

Notice that by the Bochner formula,
\begin{equation*}
\Delta |\omega|\geq -c\, |\omega|
\end{equation*}
on $D(r_0),$ where $c$ depends on the curvature bounds of $D(r_0).$ So a
mean value inequality of the following form holds for the function $%
|\omega|. $
\begin{equation}
\sup_{ D(r_0-\varepsilon)}\left\vert \omega \right\vert ^{2}\leq \frac{C}{%
V(D(r_0))}\, \int_{D(r_0)}\left\vert \omega\right\vert ^{2}
\end{equation}%
Together with (\ref{h4}), this implies that $F$ is finite dimensional as
indicated in the proof of Theorem \ref{holo}. The dimension of $F$ depends
on $q,$ $r_{0},$ the bounds of $f$ on $D\left( r_{0}\right) $ and also on a
curvature bound on $D\left( r_{0}\right).$ This proves the theorem.
\end{proof}

\bigskip

\noindent{\small DEPARTMENT OF MATHEMATICS}\newline
{\small UNIVERSITY OF CONNECTICUT}\newline
{\small STORRS, CT 06269}\newline
{\small E-mail address: ovidiu.munteanu@uconn.edu}

\bigskip

\noindent{\small SCHOOL OF MATHEMATICS}\newline
{\small UNIVERSITY OF MINNESOTA }\newline
{\small MINNEAPOLIS, MN 55455} \newline
{\small E-mail address: jiaping@math.umn.edu}

\end{document}